\documentclass[12pt]{article}

\usepackage[left=2cm,right=2cm,top=3cm,bottom=3cm]{geometry}

\usepackage{amsmath, amssymb, latexsym, amsthm}
\usepackage{fullpage}

\usepackage{tikz}
\usetikzlibrary{shapes}
\usetikzlibrary{calc,decorations.pathreplacing}

\newtheorem{theorem}{Theorem}[section] 
\newtheorem{theorem*}{Theorem} 
\newtheorem{proposition}[theorem]{Proposition} 

\newtheorem{corollary}[theorem]{Corollary}
\newtheorem{lemma}[theorem]{Lemma}

\theoremstyle{definition}

\newtheorem{observation}[theorem]{Observation}
\newtheorem{question}{Question}
\newtheorem{definition}[theorem]{Definition}
\newtheorem{example}[theorem]{Example}

\theoremstyle{remark}

{\end{enumerate}}

\newcommand{\NN}{\mathbb{N}}

\newcommand{\CL}[1]{\left\lceil #1 \right\rceil}

\newcommand{\FL}[1]{\left\lfloor #1 \right\rfloor}

\newcommand{\au}{a_u}
\newcommand{\af}{a_f}
\newcommand{\at}{a_t}
\newcommand{\A}{\mathcal{A}}

\DeclareMathOperator{\wt}{wt}

\expandafter\ifx\csname dplus\endcsname\relax \csname newbox\endcsname\dplus\fi
\expandafter\ifx\csname dplustemp\endcsname\relax
\csname newdimen\endcsname\dplustemp\fi
\setbox\dplus=\vtop{\vskip -8pt\hbox{%
    \kern .2em
    \special{pn 6}%
    \special{pa 0 50}%
    \special{pa 50 100}%
    \special{fp}%
    \special{pa 50 100}%
    \special{pa 100 50}%
    \special{fp}%
    \special{pa 100 50}%
    \special{pa 50 0}%
    \special{fp}%
    \special{pa 50 100}%
    \special{pa 50 0}%
    \special{fp}%
    \special{pa 50 0}%
    \special{pa 0 50}%
    \special{fp}%
    \special{pa 0 50}%
    \special{pa 100 50}%
    \special{fp}%
    \hbox{\vrule depth0.080in width0pt height 0pt}%
    \kern .7em
  }%
}%

\def\esub{\subseteq}
\def\nul{\varnothing}
\def\Gb{\overline{G}}

\def\VEC#1#2#3{#1_{#2},\ldots,#1_{#3}}

\def\FR{\frac}
\def\FL#1{\left\lfloor #1\right\rfloor}
\def\CL#1{\left\lceil #1\right\rceil}
\def\SE#1#2#3{\sum_{#1=#2}^{#3}}
\def\smvs{\sum_{v\in S}}
\def\Vxy{V_{x,y}}

\title{The Unit Acquisition Number of a Graph}
\author{Frederick Johnson\footnotemark[1], Anna Raleigh\footnotemark[1],\\ Paul S.\ Wenger\footnotemark[1],\ \ and Douglas B.\ West\footnotemark[2]}

\begin{document}

\maketitle

\begin{abstract}
Let $G$ be a graph with nonnegative integer weights.
A {\it unit acquisition move} transfers one unit of weight from a vertex to a
neighbor that has at least as much weight.
The {\it unit acquisition number} of a graph $G$, denoted $\au(G)$, is the
minimum size that the set of vertices with positive weight can be reduced to
via successive unit acquisition moves when starting from the configuration in
which every vertex has weight $1$.

For a graph $G$ with $n$ vertices and minimum degree $k$, we prove
$\au(G)\le (n-1)/k$, with equality for complete graphs and $C_5$.  Also
$\au(G)$ is at most the minimum size of a maximal matching in $G$, with
equality on an infinite family of graphs.  Furthermore, $\au(G)$ is bounded by
the maximum degree and by $\sqrt{n-1}$ when $G$ is an $n$-vertex tree with
diameter at most $4$.  We also construct arbitrarily large trees with maximum
degree $5$ having unit acquisition number $1$, obtain a linear-time algorithm
to compute the acquisition number of a caterpillar, and show that graphs with
diameter $2$ have unit acquisition number $1$ except for $C_5$ and the Petersen
graph.

\end{abstract}

\renewcommand{\thefootnote}{\fnsymbol{footnote}}
\footnotetext[1]{
School of Mathematical Sciences, Rochester Institute of Technology, Rochester, NY;
{\tt fwj9028@rit.edu, anr2291@rit.edu, pswsma@rit.edu}.}
\footnotetext[2]{
Departments of Mathematics, Zhejiang Normal University, China, and University of Illinois, USA ,
{\tt west@math.uiuc.edu}. Research supported by Recruitment Program of Foreign Experts, 1000 Talent Plan,
State Administration of Foreign Experts Affairs, China}
\renewcommand{\thefootnote}{\arabic{footnote}}

\baselineskip18pt

\section{Introduction}

Consider a network of cities in which troops have been deployed.  If the goal
is to airlift the troops out of the cities, then a possible protocol for moving
the troops is to first consolidate troops by having those in a city with many
troops protect a smaller number arriving from a neighboring city.  The goal is
to minimize the number of cities where an airlift is then needed to remove the
troops.  In this setting, can all of the troops move to a single city?

We model this scenario with a graph in which every vertex initially has weight
$1$.  A {\it unit acquisition move} transfers one unit of weight from a vertex
$u$ to a neighbor $v$, under the condition that before the move $v$ has at
least as much weight at $u$.  The {\it unit acquisition number} of a graph $G$,
denoted $\au(G)$, is the minimum size that the set of vertices with positive
weight can be reduced to via successive unit acquisition moves when starting
from the configuration in which every vertex has weight $1$.
Because a legal move from $u$ to $v$ can be followed by another if $u$ has
additional weight, we can equivalently transfer any integer amount of weight
when a transfer is allowed.

Acquisition numbers were introduced by Lampert and Slater~\cite{LS}.  They
referred to unit acquisition moves (and the more general integer transfers)
as {\it consolidations} and called a consolidation that transfers all the weight
from a vertex an {\it acquisition move}.  To unify the terminology across
several models, we call such a move a {\it total acquisition move}; thus
``unit'' indicates that not all of the weight need move.  Another model allows
consolidations that move fractional (non-integer) amounts of weight.  In each
case the initial distribution has weight $1$ at each vertex, and the aim is to
make the set retaining positive weight as small as possible using the specified
type of consolidation.  In addition to the unit acquisition number $\au(G)$,
the corresponding parameters are the {\it total acquisition number} $\at(G)$
and the {\it fractional acquisition number} $\af(G)$.

Up to now, the study of acquisition has focused on total acquisition and
fractional acquisition.  Lampert and Slater~\cite{LS} proved
$a_t(G)\le \FL{(n+1)/3}$ and observed that no vertex $v$ can reach weight
greater than $2^{d(v)}$, where $d(v)$ is the degree of $v$.
Slater and Wang~\cite{SW} later proved that deciding whether a graph has total
acquisition number $1$ is NP-complete, and they gave a linear-time algorithm to
determine $\at(G)$ when $G$ is a caterpillar.  Among other results,
LeSaulnier et al.~\cite{LPSWWW} characterized the trees $G$ such that
$\at(G)=1$, obtained sharp bounds on $\at(G)$ in terms of the diameter
and number of vertices when $G$ is a tree, gave bounds on $\at(G)$ when $G$ has
diameter $2$, and proved $\min\{\at(G),\at(\Gb)\}=1$, where $\Gb$ is the
complement of $G$.  LeSaulnier and West~\cite{LW} characterized the $n$-vertex
trees having the largest total acquisition number.  MacDonald, Wenger, and
Wright~\cite{MWW} determined the total acquisition numbers of most grids.
Bal et al.~\cite{BBDP} studied the total acquisition number of random graphs,
determining the threshold function for $\at(G(n,p))=1$ in the usual binomial
random graph model with independent edge probability $p$. 

Note that unit and then fractional acquisition provide more flexibility in
choosing moves than total acquisition does; thus $\af(G)\le \au(G)\le \at(G)$.
Wenger~\cite{W} determined the fractional acquisition number of every graph:
$\af(G)=1$ if $G$ is connected and $G$ is not a path or a cycle, but
$\af(P_n)=\af(C_n)=\CL{\frac n4}$ where $P_n$ and $C_n$ are the $n$-vertex path
and cycle, respectively.

We begin the study of unit acquisition number.
In Section~\ref{Sec:init}, we obtain general bounds.  Let $\delta(G)$ and
$\Delta(G)$ denote the minimum and maximum vertex degrees in $G$, respectively.
For a graph $G$ with $n$ vertices and minimum degree $k$,
we prove $\au(G)\le (n-1)/\delta(G)$ when $G$ has $n$ vertices; equality
holds if and only if $G\in\{K_n,C_5\}$, where $K_n$ is the $n$-vertex complete
graph.  Also $\au(G)$ is at most the minimum size of a maximal matching in $G$,
and this bound is sharp on infinitely many graphs with maximum degree $k$ when $k\ge4$.
Also $\au(G)\le\Delta(G)$ and $\au(G)\le\sqrt{n-1}$ when $G$ is an $n$-vertex
tree with diameter at most $4$.

In Section~\ref{Sec:unbounded}, we construct arbitrarily large trees with
maximum degree $5$ having unit acquisition number $1$.  This is very different
from total acquisition, where each vertex $v$ can receive weight at most
$2^{d(v)}$, and hence $\at(G)\ge n/32$ when $G$ is a tree with maximum degree
at most $5$. 

A {\it caterpillar} is a tree in which all non-leaf vertices lie on one path.
In Section~\ref{Sec:trees}, we characterize the caterpillars $G$ such that
$\au(G)=1$.  This leads to a linear-time algorithm to determine
$\au(G)$ when $G$ is a caterpillar; Slater and Wang~\cite{SW} found such an
algorithm for $\at(G)$ on caterpillars.

In Section~\ref{Sec:diam2}, we prove $\au(G)\le 2$ for every graph $G$ with
diameter $2$, with equality only for $C_5$ and possibly the Peterson graph.
The question from~\cite{LPSWWW} whether $\at(G)\le2$ whenever $G$ has 
diameter $2$ remains open.

We end this introduction with several open questions about unit acquisition.
Recall that deciding $\at(G)=1$ is NP-complete, while ``$\af(G)=1$?'' can be
answered in time linear in $|V(G)|$ by the results in~\cite{W}.

\begin{question}
What is the complexity of deciding whether $\au(G)=1$?
\end{question}

Although we have determined which caterpillars have unit acquisition number $1$,
the question seems more difficult for general trees.

\begin{question}
Is there a simple characterization of the trees with unit acquisition number
$1$?
\end{question}

Our construction of arbitrarily large trees with maximum degree $5$ having
unit acquisiton number $1$ suggests an obvious question.

\begin{question}
Is there a bound on the number of vertices in a graph $G$ with maximum degree
$4$ such that $\au(G)=1$?
\end{question}

We suspect that the answer to this question is yes, and that the maximum number
of vertices is somewhere around $250$.

\section{Basic results}\label{Sec:init}

We maintain the initial restriction that unit acquisition moves transfer one
unit of weight.  The additional flexibility of moving more weight in one move
does not change any results, and the restriction simplifies the analysis in
many places.  We thus view the $n$ units of weight on an $n$-vertex graph as
{\it chips}; the initial unit of weight at a vertex $v$ is the {\it chip} for
$v$, denoted $c_v$.  Finally, we call a sequence of unit acquisition moves
a {\it protocol}, and a protocol on a graph $G$ is {\it optimal} if it leaves
only $\au(G)$ vertices with positive weight.

\begin{proposition}
Every protocol is finite.
\end{proposition}

\begin{proof}
On an $n$-vertex graph $G$, the sum of the squares of the vertex weights is
initially $n$, cannot exceed $n^2$, and increases by at least $2$ with every
unit acquisition move.
\end{proof}

\begin{proposition}\label{partpath}
$\au(P_n)=\au(C_n)=\CL{n/4}$.
\end{proposition}

\begin{proof}
From~\cite{LPSWWW} and~\cite{W}, $\at(C_n)=\at(P_n)=\CL{n/4}$ and
$\af(C_n)=\af(P_n)=\CL{n/4}$.  In the introduction we observed
$\af(G)\le \au(G)\le \at(G)$ for every graph $G$.
\end{proof}

\begin{example}
In studying total acquisition or fractional acquisition numbers, it suffices
to consider acyclic graphs.  In total acquisition, weight can leave a vertex
at most once, after which no weight can move to it; hence the set of edges used
in a total acquisition protocol is acyclic.  In fractional acquisition, any
connected graph having maximum degree at least $3$ has a spanning tree with
maximum degree at least $3$, which suffices for fractional acquisition number
$1$, while $a_f(G)=\CL{|V(G)|/4}$ when $G\in\{P_n,C_n\}$~\cite{W}.  However,
optimal protocols for unit acquisition may need edge sets with cycles.  For
example, if $G$ is the graph in Figure~\ref{fig:cycle}, then $\au(G)=1$, but
$\au(G-e)=2$ when $e$ is any edge in $G$.
\end{example}

\begin{figure}[hbt]
\centering
\begin{tikzpicture}
\fill (0,0) circle (0.15);
\fill (1,0) circle (0.15);
\fill (2,0) circle (0.15);
\fill (3,0) circle (0.15);
\fill (4,0) circle (0.15);
\fill (1.5,1) circle (0.15);
\fill (1.5,2) circle (0.15);

\draw[line width=2] (0,0) -- (4,0);
\draw[line width=2] (1,0) -- (1.5,1) -- (1.5,2);
\draw[line width=2] (2,0) -- (1.5,1);
\end{tikzpicture}
\caption{A minimal graph with unit acquisition number $1$.\label{fig:cycle}}
\end{figure}
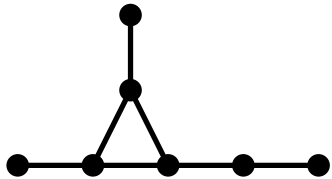

In a rooted tree, the {\it parent} of a non-root vertex $v$ is its neighbor
along the path from $v$ to the root.  An \textit{ascending tree} is a rooted
weighted tree such that the weight of every leaf is at most the weight of its
parent, and the weight of every non-root non-leaf vertex is strictly less than
the weight of its parent.

\begin{observation}
If a unit aquisition protocol on a tree $T$ can turn it into 
an ascending tree, then $\au(T) = 1$.
\end{observation}

\begin{proof}
The weight on a leaf in $T$ can be moved to the root one unit at a time.
Repeating this moves all weight in the tree to the root.
\end{proof}

Let $N_H(v)$ denote the set of neighbors of a vertex $v$ in a graph $H$. 

\begin{lemma}\label{level2}
If a graph $G$ has a spanning tree with every vertex having distance at most
$2$ from the root, and some edge joins vertices at distance $2$ from the root,
then $\au(G)\le2$.
\end{lemma}
\begin{proof}
Let $v$ be the root, with children $\VEC x1k$.  We may assume that a child $z$
of $x_1$ has a neighbor $w$ with distance $2$ from $v$ in the tree.
First move $c_w$ to $z$.  Now in two steps $z$ can acquire the weight from
each other child of $x_1$.  After $z$ acquires all these chips, move $c_{x_1}$
from $x_1$ to $v$.  Now all weight not on $z$ lies in an increasing tree
with root $v$.
\end{proof}

\begin{proposition}
If $G$ is an $n$-vertex graph, then $\au(G)\leq({n-1})/\delta(G)$, with
equality if and only if $G\in\{K_n,C_5\}$.
\end{proposition}

\begin{proof}
Clearly $\au(K_n)=1$, and $\au(C_5)=2$ is easy to check, so these achieve
equality.

For general $G$, let $S$ be a largest set of vertices such that the distance
between any two is at least $3$.  Let $S=\{v_1,\ldots,v_m\}$.
Any vertex not in $S$ has distance at most $2$ from $S$.  Partition $V(G)$ into
sets $V_1,\ldots,V_m$ with $v_i\in V_i$ by assigning each vertex not in $S$ to
a nearest vertex in $S$.  For each $i$, the set $V_i$ is the vertex set of a
tree $T_i$ contained in $G$ such that $d_{T_i}(v_i)=d_G(v_i)$ and each vertex
of $T_i$ has distance at most $2$ from $v_i$.

Let $x_i$ be a vertex of $N_{T_i}(v_i)$ with least degree in $T_i$, and let
$X_i=N_{T_i}(x_i)-\{v_i\}$.  Move $c_{x_i}$ from $x_i$ to $v_i$, making
$T_i-X_i$ an ascending tree in $G$ rooted at $v_i$.  All weight on $T_i-X_i$
can now move to $v_i$, so $\au(T_i)\le d_{T_i}(x_i)$.  By the choice of $S$, we
have $N_G(v_i)\esub V_i$.  Hence
\begin{equation}\label{eq1}
a_u(T_i)\le
d_{T_i}(x_i)\le\FR{|V(T_i)|-1}{d_{G}(v_i)}\le\FR{|V(T_i)|-1}{\delta(G)}.
\end{equation}
Therefore,
\begin{equation}\label{eq2}
\au(G)\le\sum_{i=1}^m \au(T_i)\le\sum_{i=1}^m
\FR{|V(T_i)|-1}{\delta(G)}\le\FR{n-1}{\delta(G)}.
\end{equation}


Now suppose $\au(G)=\FR{n-1}{\delta(G)}$; note that $\delta(G)>1$ is needed
unless $G=K_2$.  Equality always holds for $K_n$, so assume $G\ne K_n$.
Equality at the end of (\ref{eq2}) requires $m=1$, which by the definition of
$S$ requires diameter $2$.  Equality at the end of (\ref{eq1}) requires
$d_G(v_1)=\delta(G)$.  Equality in the middle of (\ref{eq1}) requires that all
neighbors of $v_1$ have the same degree in $T_1$.  If some vertex at distance
$2$ from $v_1$ has two neighbors in $N(v_1)$, then $T_1$ can be changed to
reduce the bound.  Hence each vertex at distance $2$ from $v$ has another
neighbor at distance $2$, since $\delta(G)>1$.  Now Lemma~\ref{level2}
implies $\au(G)\le2$.

Let $k=\delta(G)$.  Equality in (2) now requires $(n-1)/k=2$.  This also equals
$d_{T_1}(x_i)$, so each neighbor of $v$ has one child in $T_1$.  Recall also
that each vertex at level $2$ has no neighbor at level $1$ other than its
parent.  Let $X=N_G(v)$.  Since $v$ has minimum degree in $G$, every vertex of
$X$ is thus adjacent to all except possibly one other vertex of $X$.  If $k>2$,
then $X$ can acquire all the weight from level $2$, followed by accumulating
all except $c_v$ at one vertex of $X$, which then acquires $c_v$ to reach
$\au(G)=1$.  Hence we may assume $k=2$, and now avoiding $\au(G)=1$ requires
$G=C_5$. 
\end{proof}

With a bit more work, Lemma~\ref{level2} can be used to show that
$\au(G)\le2$ when $G$ has diameter $2$.  We omit this proof, because in
Section~\ref{Sec:diam2} we prove the stronger result that $\au(G)=1$ when $G$
has diameter $2$ and is not $C_5$ or the Petersen graph.

Our next upper bound is sharp more often.  We begin with a lemma used to prove
equality in the bound for many graphs.  Two chips {\it meet} when they reach
the same vertex.

\begin{lemma}\label{cut}
Given vertices $u$ and $v$ in a graph $G$, let $S$ be a minimal $u,v$-cut in
$G$.  If every vertex in $S$ has degree $2$ in $G$, and $u$ and $v$ have no
neighbors in $S$, then the chips from $u$ and $v$ cannot meet
via unit acquisition moves.
\end{lemma}

\begin{proof}
Because $S$ is a minimal $u,v$-cut and every vertex of $S$ has degree $2$, no
edges are induced by $S$.  If a protocol moves a chip along an edge with
endpoint $x\in S$, then the first move along an edge incident to $x$ reduces
the weight of $x$ to $0$ or brings $x$ the chip from a neighbor (which then has
weight $0$).  Since $u,v\notin N(x)$, the chip transferred is not $c_u$ or
$c_v$.  Also, once an endpoint of an edge has weight $0$, the edge cannot be
used again.

For each $x\in S$, delete the first edge at $x$ used by the protocol, or $x$
itself if no edge at $x$ is used.  This leaves $u$ and $v$ in distinct
components, and no move can transfer $c_u$ or $c_v$ from one
component to the other.  Hence these two chips cannot meet.
\end{proof}

\begin{proposition}
If $G$ is a connected graph, then $\au(G)$ is at most the minimum size of a
maximal matching in $G$.  For $m,k\in\NN$ with $k\ge4$, some graph $G_{m,k}$
with maximum degree $k$ has a maximal matching of size $m$ and
$\au(G_{m,k})=m$, achieving equality in the bound.
\end{proposition}

\begin{proof}
Let $M$ be a smallest maximal matching in $G$.  The vertices of $G$ can be
partitioned into $|M|$ sets that induce trees with diameter at most $3$, each
of which has unit acquisition number $1$.  Thus $\au(G) \le |M|$.

To construct $G_{m,k}$, first let $H_k$ be the tree with $2k$ vertices having
two central vertices of degree $k$ and $2k-2$ leaves.  Form $G_{m,k}$ from
$m$ disjoint copies of $H_k$ as follows.  For $1\le i\le k-1$, let $x$ and $y$
be leaf neighbors of the two central vertices in the $i$th copy of $H_k$, such
that $x$ and $y$ still have degree $1$ in the graph being formed.
Also choose leaf neighbors $x'$ and $y'$ of the two central vertices in the
$(i+1)$th copy of $H_k$.  Merge $x$ with $x'$ and $y$ with $y'$, forming two
vertices of degree $2$.  Note that the resulting graph $G_{m,k}$ decomposes
into $M$ copies of $H_k$, and the $m-1$ pairs of vertices formed in the 
merging steps are vertex cuts in which each vertex has degree $2$.
Figure~\ref{Gm4} shows $G_{4,5}$.

The central vertices of the copies of $H_k$ form a maximal matching of size
$m$ in $G_{m,k}$.  Each vertex covered by this matching has a leaf neighbor
in $G_{m,k}$.  Let $u_i$ be such a leaf in the $i$th copy of $H_k$.
Vertices $u_i$ and $u_j$ are separated by a vertex cut containing no neighbor
of $u_i$ or $u_j$ in which each vertex has degree $2$.  By Lemma~\ref{cut},
no two chips in $\{c_{u_1},\dots,c_{u_m}\}$ can meet under
any protocol.  Hence $\au(G_{m,k})\ge m$, and equality holds.
\end{proof}

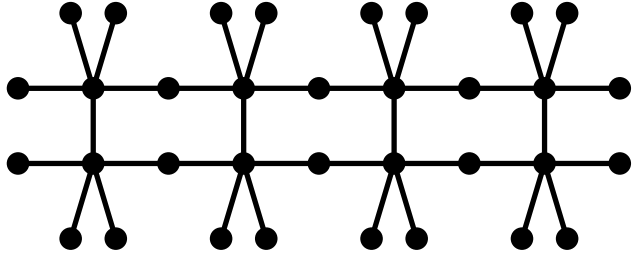
\begin{figure}[hbt]
 \begin{center}

\begin{tikzpicture}
\fill (-.3,0) circle (0.15);
\fill (1.7,0) circle (0.15);
\fill (3.7,0) circle (0.15);
\fill (5.7,0) circle (0.15);
\fill (.3,0) circle (0.15);
\fill (2.3,0) circle (0.15);
\fill (4.3,0) circle (0.15);
\fill (6.3,0) circle (0.15);
\fill (-1,1) circle (0.15);
\fill (0,1) circle (0.15);
\fill (1,1) circle (0.15);
\fill (2,1) circle (0.15);
\fill (3,1) circle (0.15);
\fill (4,1) circle (0.15);
\fill (5,1) circle (0.15);
\fill (6,1) circle (0.15);
\fill (7,1) circle (0.15);
\fill (-1,2) circle (0.15);
\fill (0,2) circle (0.15);
\fill (1,2) circle (0.15);
\fill (2,2) circle (0.15);
\fill (3,2) circle (0.15);
\fill (4,2) circle (0.15);
\fill (5,2) circle (0.15);
\fill (6,2) circle (0.15);
\fill (7,2) circle (0.15);
\fill (-.3,3) circle (0.15);
\fill (1.7,3) circle (0.15);
\fill (3.7,3) circle (0.15);
\fill (5.7,3) circle (0.15);
\fill (.3,3) circle (0.15);
\fill (2.3,3) circle (0.15);
\fill (4.3,3) circle (0.15);
\fill (6.3,3) circle (0.15);

\draw[line width=2] (-.3,0) -- (0,1) --(0,2)--(-.3,3);
\draw[line width=2] (.3,0) -- (0,1);
\draw[line width=2] (0,2) -- (.3,3);
\draw[line width=2] (1.7,0) -- (2,1) --(2,2)--(1.7,3);
\draw[line width=2] (2.3,0) -- (2,1);
\draw[line width=2] (2,2) -- (2.3,3);
\draw[line width=2] (3.7,0) -- (4,1) --(4,2)--(3.7,3);
\draw[line width=2] (4.3,0) -- (4,1);
\draw[line width=2] (4,2) -- (4.3,3);
\draw[line width=2] (5.7,0) -- (6,1) --(6,2)--(5.7,3);
\draw[line width=2] (6.3,0) -- (6,1);
\draw[line width=2] (6,2) -- (6.3,3);

\draw[line width=2] (-1,1) -- (7,1);
\draw[line width=2] (-1,2) -- (7,2);

\end{tikzpicture}

 \end{center}
\caption{The graph $G_{4,5}$.\label{Gm4}}
\end{figure}

LeSaulnier et al.~\cite{LPSWWW} proved $\at(T)\le \sqrt{n\lg n}$ when $T$
is an $n$-vertex tree with diameter $4$ and constructed an $n$-vertex tree
$T_n$ with diameter $4$ such that $\at(T_n)\ge (1-o(1))\sqrt{(n/2)\lg n}$.
For unit acquisition number, a stronger bound holds.

\begin{proposition}
If $T$ is a tree of diameter at most $4$, then
$\au(T)\leq\sqrt{n-1}\le\Delta(T)$.  Equality holds in the first inequality
only when $n-1$ is a square and $T$ is the tree of diameter $4$ such that the
central vertex and all its neighbors have degree $\sqrt{n-1}$.
\end{proposition}

\begin{proof}
Since trees of diameter at most $3$ have unit acquisition number $1$,
we may assume that $T$ has diameter $4$.  Let $v$ be the central vertex.
Let $k=\sqrt{n-1}$.  Note that when $T$ has diameter at most $4$, always
$\Delta(T)\ge k$, since otherwise $n$ vertices cannot be found.

When $d_T(v)\le k$, move all weight to $N_T(v)$ to obtain
$\au(T)\le d_T(v)\le k$.  Otherwise, move weight $1$ to $v$ from the neighbor
$u$ of $v$ having least degree.  Now the tree formed by deleting $u$ and its
leaf neighbors is ascending, so $\au(T)\leq 1+d_T(u)-1$.  Since in this case
$d_T(v)>k$, by the pigeonhole principle $d_T(u)<k$.  Thus $\au(T)\le k$.

Avoiding $\au(T)< k$ first requires $d_T(v)\ge k$.  Since
$n\ge 1+d_T(v)d_T(u)$, having $d_T(v)\ge k$ requires $d_T(u)\le k$.
Since $\au(T)\le d_T(u)$, avoiding $\au(T)< k$ requires $k=d_T(u)=d_T(v)$, and
$T$ is the tree specified in the theorem statement.

For this tree $T$, if no weight moves to the root, then weight remains in $k$
disjoint subtrees.  If a chip moves to the root, then the first such chip
leaves $k-1$ isolated vertices with positive weight.  Hence $\au(T)\ge k$.
\end{proof}

\section{Trees $T$ with $\Delta(T)=5$ and $\au(T)=1$}\label{Sec:unbounded}

For total acquisition, Lampert and Slater~\cite{LS} proved
$\at(G)\geq |V(G)|/2^{\Delta(G)}$ by observing that a vertex $v$ cannot
acquire weight more than $2^{d_G(v)}$ via total acquisition moves.
For unit acquisition, no analogous result can be proved, since there is
no bound on the amount of weight that a vertex in a tree of maximum degree 5
can acquire.

\begin{theorem}\label{thm:UnboundedWeight}
For $d\in\NN$, there is a tree $T_d$ with maximum degree $5$ in which some
vertex can acquire weight at least $d$ via a unit acquisition protocol
and $\au(T_d)=1$.
\end{theorem}

\begin{proof}
We inductively construct a rooted tree $T_{d}$ with maximum degree $5$ in which
we can move all the weight to an ascending tree.  We use $T'_d$ to denote the
ascending version of $T_d$ after this partial protocol.  Since $T'_d$ is
ascending, $\au(T_d)=1$.  

We construct $T_d$ from $T'_{d-1}$ by adding leaves.  Since $T_d$ contains
$T_{d-1}$, the partial protocol can be followed on $T_{d-1}$ to obtain
$T'_{d-1}$ within $T_d$.  Further unit acquisition moves will then produce
$T'_d$.  The tree $T_d$ will have $d$ levels, with the root as the first.
Thus $T_1$ consists of only the root.  We also specify $T_2$ explicitly; it
consists of the root plus five children.  Let $T'_2$ be the ascending tree
produced by moving the chip from one child to the root (see
Figure~\ref{fig:AscTrees}).

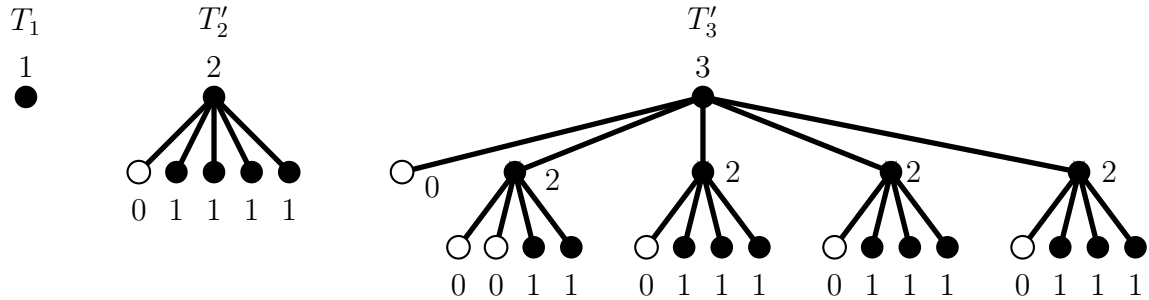
\begin{figure}[hbt]
\begin{center}

\begin{tikzpicture}

\fill (2.5,2) circle (0.15);
\node at (2.5,3) {$T_1$};
\node at (2.5,2.4) {$1$};

\draw [line width=2] (4,1)--(5,2)--(6,1);
\draw [line width=2] (4.5,1)--(5,2)--(5.5,1);
\draw [line width=2] (5,1)--(5,2);

\fill (5,2) circle (0.15);
\draw [fill=white,thick] (4,1) circle (0.15);
\fill (4.5,1) circle (0.15);
\fill (5,1) circle (0.15);
\fill (5.5,1) circle (0.15);
\fill (6,1) circle (0.15);

\node at (5,3) {$T'_2$};
\node at (5,2.4) {$2$};
\node at (4,.5) {$0$};
\node at (4.5,.5) {$1$};
\node at (5,.5) {$1$};
\node at (5.5,.5) {$1$};
\node at (6,.5) {$1$};

\draw [line width=2] (7.5,1) -- (11.5,2);
\draw [line width=2] (9,1) -- (11.5,2);
\draw [line width=2] (11.5,1) -- (11.5,2);
\draw [line width=2] (14,1) -- (11.5,2);
\draw [line width=2] (16.5,1) -- (11.5,2);

\draw [line width=2] (8.25,0) -- (9,1) -- (8.75,0);
\draw [line width=2] (9.25,0) -- (9,1) -- (9.75,0);

\draw [line width=2] (10.75,0) -- (11.5,1) -- (11.25,0);
\draw [line width=2] (11.75,0) -- (11.5,1) -- (12.25,0);

\draw [line width=2] (13.25,0) -- (14,1) -- (13.75,0);
\draw [line width=2] (14.25,0) -- (14,1) -- (14.75,0);

\draw [line width=2] (15.75,0) -- (16.5,1) -- (16.25,0);
\draw [line width=2] (16.75,0) -- (16.5,1) -- (17.25,0);

\fill (11.5,2) circle (0.15);
\draw [fill=white,thick] (7.5,1) circle (0.15);

\fill (9,1) circle (0.15);
\draw [fill=white,thick] (8.25,0) circle (0.15);
\draw [fill=white,thick] (8.75,0) circle (0.15);
\fill (9.25,0) circle (0.15);
\fill (9.75,0) circle (0.15);

\fill (11.5,1) circle (0.15);
\draw [fill=white,thick] (10.75,0) circle (0.15);
\fill (11.25,0) circle (0.15);
\fill (11.75,0) circle (0.15);
\fill (12.25,0) circle (0.15);

\fill (14,1) circle (0.15);
\draw [fill=white,thick] (13.25,0) circle (0.15);
\fill (13.75,0) circle (0.15);
\fill (14.25,0) circle (0.15);
\fill (14.75,0) circle (0.15);

\fill (16.5,1) circle (0.15);
\draw [fill=white,thick] (15.75,0) circle (0.15);
\fill (16.25,0) circle (0.15);
\fill (16.75,0) circle (0.15);
\fill (17.25,0) circle (0.15);

\node at (11.5,3) {$T'_3$};
\node at (11.5,2.4) {$3$};
\node at (7.9,.8) {$0$};
\node at (9.5,.9) {$2$};
\node at (11.9,1) {$2$};
\node at (14.3,1) {$2$};
\node at (16.9,1) {$2$};

\node at (8.25,-.5) {$0$};
\node at (8.75,-.5) {$0$};
\node at (9.25,-.5) {$1$};
\node at (9.75,-.5) {$1$};

\node at (10.75,-.5) {$0$};
\node at (11.25,-.5) {$1$};
\node at (11.75,-.5) {$1$};
\node at (12.25,-.5) {$1$};

\node at (13.25,-.5) {$0$};
\node at (13.75,-.5) {$1$};
\node at (14.25,-.5) {$1$};
\node at (14.75,-.5) {$1$};

\node at (15.75,-.5) {$0$};
\node at (16.25,-.5) {$1$};
\node at (16.75,-.5) {$1$};
\node at (17.25,-.5) {$1$};

\end{tikzpicture}
\end{center}
\caption{The first three trees for Theorem~\ref{thm:UnboundedWeight},
converted to ascending trees.}\label{fig:AscTrees}
\end{figure}

We call a vertex with positive weight in $T'_d$ \textit{active}, except that
some leaves retaining weight $1$ may be designated inactive.  The root vertex
is at level $1$; the leaves are at level $d$.  In $T'_d$, the active vertices
at level $i$ have weight $d+1-i$.  Let $a_d$ denote the number of active
vertices at level $d$ in $T'_d$, so $a_2=4$.

For $d\ge3$, we construct $T_d$ from $T'_{d-1}$ by appending four leaves at
each active vertex on level $d-1$ and then viewing each vertex as starting
with weight $1$.  To convert $T_d$ to $T'_d$, first perform the protocol
on the copy of $T_{d-1}$ within $T_d$ formed by levels $1$ through $d-1$.
By the induction hypothesis, this puts weight $d-i$ at each active vertex in
level $i$, for $i<d$.

We now want to use chips from the leaves to increase the weight by $1$ at
(most) active non-leaf vertices.  For $i$ from $1$ through $d-1$ successively,
for each active vertex $u$ at level $i$ currently having an active leaf below
it on level $d$, choose such a leaf $x$ and move $c_x$ up the path through the
tree to reach $u$.  This is possible, because the tree remains ascending
throughout the process.  See $T_3$ in Figure~\ref{fig:AscTrees}.

If there is no such leaf below $u$, then we instead choose some other 
remaining leaf $x$ that has weight $1$, arbitrarily, and designate $x$
inactive.  The tree remains ascending, because the parent of this leaf has
weight at least $1$.  More importantly, when we grow $T_{d+1}$ we will not
add leaves below $x$, since $x$ is inactive.  Hence it causes no difficulty
if the parent of $x$ also ends the process with weight $1$.

The process thus can be completed if the number of leaves added in forming
$T_d$ is at least the total number of active vertices in $T'_{d-1}$.
The value $a_i$ is the number of active vertices left at level $i$ when
$T'_i$ is formed, and this always remains the number of active vertices at
level $i$.  For $d\ge3$, we thus have 
$$a_{d} = 4a_{d-1} - \sum_{i=1}^{d-1}a_{i}.$$
We can continue growing larger trees with the desired properties if $a_d>0$ for
$d\ge1$.

With $a_1=1$ and $a_2=4$, writing the recurrence as $4a_{d-1}=\SE i1d a_i$
for $d\ge3$ yields $a_3=11$.  The difference of two consecutive instances of
the recurrence yields $a_d = 4a_{d-1} - 4a_{d-2}$ for $d\ge4$, with $a_2=4$ and
$a_3=11$.  The solution $a_d = (3d + 5)2^{d-2}$ for $d\ge2$ is easily checked
by induction.  As desired, $a_d>0$ for all $d$, which completes the proof.
\end{proof}

For graphs with maximum degree $1$, $2$, or $3$, straightforward case analysis
shows that a vertex can acquire weight at most $2$, $4$, or $10$, respectively.
With maximum degree $4$, growing three leaves at active vertices, the
corresponding recurrence in the method above is $a_d = 3(a_{d-1} - a_{d-2})$
for $d\ge4$, with $a_2 =3$ and $a_3=5$.  Also $a_4=6$, $a_5=3$, and $a_6=-9$,
so this construction does not grow beyond depth $5$, since the nine vertices
generated at level $6$ do not suffice to augment the higher active vertices.
There are $56$ vertices at this point, so a vertex an acquire weight $56$.
With more careful analysis it may be possible obtain a bound on the number of
vertices in a tree with unit acquisition number $1$ and maximum degree $4$.

\section{Unit Acquisition on Caterpillars}\label{Sec:trees}


Toward the further understanding of unit acquisition on trees,
in this section we characterize the caterpillars with $\au(T)=1$ and give a
linear-time algorithm to compute unit acquisition number on caterpillars.

\begin{definition}
The path obtained by deleting the leaves of a caterpillar is called the
{\it spine} of the caterpillar.  Given a caterpillar $T$, let
$v_0,\ldots, v_{k+1}$ denote the vertices in the spine of $T$, indexed in order.
The vertices $v_1,\ldots,v_k$ are the {\it internal vertices} of the spine of
$T$.  In particular, note that $v_0$ and $v_{k+1}$ are not leaves of $T$.

For $v\in V(T)$, let $d'(v)$ denote the number of leaf neighbors of $v$.
For $s\in\NN$, let $\ell(s)=\SE i1s\CL{i/2}=\CL{\FR{s+1}{2}}\FL{\FR{s+1}{2}}$.
During a protocol, let $\wt(v)$ be the current weight on $v$.

Let $S$ be a set of $s$ consecutive internal vertices on the spine of $T$,
with $S=\{v_{i},\ldots,v_{i+s-1}\}$.  Let the {\it pyramid} of $S$ be a set of
cells arranged so that $\min\{j-i+1,i+s-j\}$ cells are stacked above $v_j$
(see Figure~\ref{fig:pyramid}).  Let $b_{j,1},\ldots,b_{j,\min\{j-i+1,i+s-j\}}$
denote the cells above $v_j$.  Counting columns moving in from both ends
shows that the pyramid of $S$ has $\ell(s)$ cells.
\end{definition}

\begin{figure}[hbt]
\centering
\begin{tikzpicture}
\foreach \x in {17,...,24}
	\fill (\x,-.5) circle (0.15);
\node at (17,-1) {$v_{2}$};
\node at (18,-1) {$v_{3}$};
\node at (19,-1) {$v_{4}$};
\node at (20,-1) {$v_{5}$};
\node at (21,-1) {$v_{6}$};
\node at (22,-1) {$v_{7}$};
\node at (23,-1) {$v_{8}$};
\node at (24,-1) {$v_{9}$};
\draw [line width = 2] (17,-.5) -- (24,-.5);
\draw (16.5,0) -- (24.5,0) -- (24.5,1) -- (16.5,1) -- (16.5,0);
\draw (17.5,0) -- (17.5,2) -- (23.5,2) -- (23.5,0);
\draw (18.5,0) -- (18.5,3) -- (22.5,3) -- (22.5,0);
\draw (19.5,0) -- (19.5,4) -- (21.5,4) -- (21.5,0);
\draw (20.5,0) -- (20.5,4);
\node at (17,.5) {$b_{2,1}$};
\node at (18,.5) {$b_{3,1}$};
\node at (18,1.5) {$b_{3,2}$};
\node at (19,.5) {$b_{4,1}$};
\node at (19,1.5) {$b_{4,2}$};
\node at (19,2.5) {$b_{4,3}$};
\node at (20,.5) {$b_{5,1}$};
\node at (20,1.5) {$b_{5,2}$};
\node at (20,2.5) {$b_{5,3}$};
\node at (20,3.5) {$b_{5,4}$};
\node at (21,.5) {$b_{6,1}$};
\node at (21,1.5) {$b_{6,2}$};
\node at (21,2.5) {$b_{6,3}$};
\node at (21,3.5) {$b_{6,4}$};
\node at (22,.5) {$b_{7,1}$};
\node at (22,1.5) {$b_{7,2}$};
\node at (22,2.5) {$b_{7,3}$};
\node at (23,.5) {$b_{8,1}$};
\node at (23,1.5) {$b_{8,2}$};
\node at (24,.5) {$b_{9,1}$};

\node at (9,-1) {$v_{1}$};
\node at (10,-1) {$v_{2}$};
\node at (11,-1) {$v_{3}$};
\node at (12,-1) {$v_{4}$};
\node at (13,-1) {$v_{5}$};
\node at (14,-1) {$v_{6}$};
\node at (15,-1) {$v_{7}$};
\draw [line width = 2] (9,-.5) -- (15,-.5);
\foreach \x in {9,...,15}
	\fill (\x,-.5) circle (0.15);
\draw (8.5,0) -- (15.5,0) -- (15.5,1) -- (8.5,1) -- (8.5,0);
\draw (9.5,0) -- (9.5,2) -- (14.5,2) -- (14.5,0);
\draw (10.5,0) -- (10.5,3) -- (13.5,3) -- (13.5,0);
\draw (11.5,0) -- (11.5,4) -- (12.5,4) -- (12.5,0);
\node at (9,.5) {$b_{1,1}$};
\node at (10,.5) {$b_{2,1}$};
\node at (10,1.5) {$b_{2,2}$};
\node at (11,.5) {$b_{3,1}$};
\node at (11,1.5) {$b_{3,2}$};
\node at (11,2.5) {$b_{3,3}$};
\node at (12,.5) {$b_{4,1}$};
\node at (12,1.5) {$b_{4,2}$};
\node at (12,2.5) {$b_{4,3}$};
\node at (12,3.5) {$b_{4,4}$};
\node at (13,.5) {$b_{5,1}$};
\node at (13,1.5) {$b_{5,2}$};
\node at (13,2.5) {$b_{5,3}$};
\node at (14,.5) {$b_{6,1}$};
\node at (14,1.5) {$b_{6,2}$};
\node at (15,.5) {$b_{7,1}$};

\end{tikzpicture}
\caption{Pyramids of lengths $7$ and $8$ on the spine of a
caterpillar.}\label{fig:pyramid}

\end{figure}
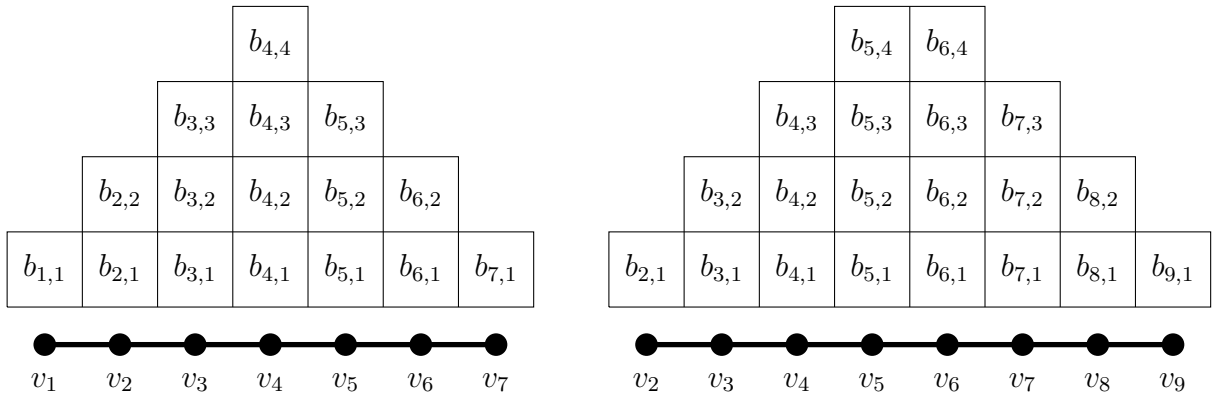

\begin{example}\label{P5}
In $P_5$ and in the caterpillar $J$ obtained from $P_5$ by appending a leaf
at the central vertex, the spine has three vertices, which we label
$v_0,v_1,v_2$ (so $k=1$).  The condition in the theorem below requires
$d'(v)\ge 1$ for every internal spine vertex, so $\au(P_5)>1$.  Indeed,
applying Lemma~\ref{cut} to leaf neighbors of the ends of the spine also
requires each internal spine vertex to have at least one leaf neighbor when
$\au(T)=1$.  Since $J$ has only one internal spine vertex, and $d'(v_1)=1$,
the condition below is satisfied, and $\au(J)=1$.
\end{example}

\begin{theorem}\label{thm:caterpillar}
Let $T$ be a caterpillar with spine vertices $v_0,v_1,\ldots,v_k,v_{k+1}$ in
order.  The unit acquisition number of $T$ is $1$ if and only if for
$1\le s\le k$, every set $S$ of $s$ consecutive internal vertices on the spine
satisfies $\sum_{v\in S}d'(v)\ge \ell(s)=\CL{\frac{s+1}{2}}\FL{\frac{s+1}{2}}$.
\end{theorem}

Necessity and sufficiency of the condition both take some work, so we
separate these proofs into two items.

\begin{theorem}
The condition on $S$ in Theorem~\ref{thm:caterpillar} is necessary for
$\au(T)=1$.
\end{theorem}
\begin{proof}
Let $S=\{\VEC vj{j+s-1}\}$.  To simplify notation, let $u_i=v_{i+j-1}$ for
$0\le i\le s+1$, so $S=\{\VEC u1s\}$.  Assume $\au(T)=1$.  As noted in
Example~\ref{P5}, $d'(u_i)\ge1$ for all $i$, by Lemma~\ref{cut}.  Hence $u_0$
has a leaf neighbor $x$, and $u_{s+1}$ has a leaf neighbor $y$, and in some
protocol $\A$ the chips $c_x$ and $c_y$ must meet.  We prove
$\smvs d'(v)\ge\ell(s)$.

Since chips in fact are indistinguishable, we may assume that at any point in
$\mathcal A$ the chips on a vertex $v$ are listed from bottom to top in their
order of arrival at $v$.  We may also assume that the a chip moved off a vertex
is the most recent chip that arrived there, again since chips are
indistinguishable.  In particular, the bottom chip on $v$ is always $c_v$.
Also, since $\au(T)=1$ requires all chips eventually to be on one vertex, it
requires $c_x$ and $c_y$ to meet even under this restriction on the movement of
chips.

Since a chip arriving at $v$ will be placed on the top, by the weight rule for
moves the height of a chip in its current stack strictly increases each time
it moves.  Since $c_x$ and $c_y$ start two steps from $S$, this
implies that they can never occupy cells in the pyramid over $S$.  Indeed, since
they start two steps away from $S$, when on a vertex of $S$ they must be at 
least one step above the pyramid (here we consider the chip on $u_j$ to be at the point $(j,0)$).

If a chip $c_v$ ever moves from $v$, then thereafter $\wt(v)=0$.  If $c_v$
moves from $v$ when $c_x$ and $c_y$ are separated by $v$, then $c_x$ and $c_y$
cannot meet under $\A$.  We will consider only protocols that move no such
chips before $c_x$ and $c_y$ meet.  This also requires that leaf neighbors of
spine vertices between $c_x$ and $c_y$ always have weight at most $1$.

For $1\le i\le s$, let $\mu_i=\min\{i,s+1-i\}$; note that $\mu_i$ is the number
of cells above $u_i$ in the pyramid over $S$.  At any point in $\A$, let
$g(u_i)=\max\{\mu_i-\wt(u_i)+1,0\}$ for $u_i\in S$.  Viewing chips at $u_i$ as
filling cells above $u_i$ in the pyramid over $S$, $g(u_i)$ gives the number of
empty cells above $u_i$ (there may also be chips above the pyramid when there
are no empty cells above a vertex).  Recall that $c_x$ and $c_y$ cannot occupy
cells inside the pyramid over $S$.

During $\A$, let $\Vxy$ denote the set of vertices in $S$ internal to the path
in $T$ joining the current locations of $c_x$ and $c_y$.
Let $l$ denote the number of chips on leaf neighbors of vertices in $\Vxy$.
As noted earlier, such leaves have at most one chip, so $l$ equals the
number of leaf neighbors of vertices in $\Vxy$ having weight $1$.

At a given time in $\A$, let $h=\sum_{u_i\in\Vxy}g(u_i)-l$.
Under the assumption $\smvs d'(v)<\ell(s)$, we will show by induction on the
number of moves in $\A$ that $h>0$ and $\Vxy\ne\nul$ throughout $\A$.  With
$\Vxy$ remaining nonempty, $c_x$ and $c_y$ never meet.
Initially, $\Vxy=S$ and $g(u_i)=\mu_i$ for $u_i\in S$; also $l=\smvs d'(v)$.
Thus $h=\ell(s)-l>0$ by assumption, and $\Vxy\ne\nul$.

Suppose that $h>0$ and $\Vxy\ne\nul$ at some point in $\A$.  Consider the next
move in $\A$.  
To decrease $h$, either $g(u_i)$ must decrease for some $u_i\in V_{x,y}$, or
$l$ must increase.  If $g(u_i)$ decreases, then some cell $b_{i,r}$ in the
column over $u_i$ (between $c_x$ and $c_y$) becomes filled.  Since no vertex
of weight $0$ can lie between $c_x$ and $c_y$, we have $i\notin\{1,s\}$.

A chip that moves to $b_{i,r}$ must come from a leaf neighbor of $u_i$ with
weight $1$ or from $b_{i\pm 1,r'}$ with $r'<r$.  If the chip moves to $b_{i,r}$
from a leaf neighbor of $u_i$, then $g(u_i)$ decreases by $1$, and $l$
decreases by $1$, so $h$ and $\Vxy$ are unchanged.  If the chip moves to
$b_{i,r}$ from $b_{i\pm 1,r'}$, then let $u_{i'}$ be the vertex contributing
the chip.  By our restrictions on chip movement, the weight on $u_{i'}$ prior
to the move is $r'$, and $b_{i',r'}$ is empty after the move.
Also, since $r'\le\mu_i$, neither $c_x$ nor $c_y$ can be on $u_{i'}$.
This yields $i,i'\in\Vxy$, so $\sum_{u_i\in V_{x,y}}g_{u_i}-l$ is unchanged,
and $\Vxy$ and $h$ are both unchanged by this move.

Therefore a decrease in $h$ must come from $l$ increasing.  This requires
moving $c_x$ and $c_y$ away from each other.  If such movement does not enlarge
$\Vxy$, then $l$ does not increase and $h$ does not change.  Hence we may
assume that $c_y$ moves to the right from $u_j$ to $u_{j+1}$, adding $u_j$ to
$\Vxy$, with $u_{j}$ still having $l_j$ leaf neighbors with weight $1$.

Since $y$ began to the right of $S$, there was a most recent time when $c_y$
moved from $u_{j+1}$ to $u_{j}$.  At that point, $h$ grew by at least $l_j$,
since then $g(u_{j+1}) = g(u_i)=0$ and there were at least $l_j$ leaves of
$u_j$ with weight $1$.  This increase of $l_j$ has not been counted in any
of the moves analyzed above.  Furthermore, between that move and when $y$ moves
back, no chip from a leaf neighbor of $u_j$ can move into a cell in the pyramid
over $S$, because it would have to land above $c_y$, which is already outside
the pyramid.  Hence the bonus contribution of $l_j$ to $h$ persists until
$c_y$ moves away, keeping $h$ still positive after that move.

Hence $h$ remains positive and $\Vxy$ remains nonempty, as desired.
\end{proof}

\begin{theorem}
If $\sum_{v\in S}d'(v)\ge\ell(s)$ for all $S$ specified in
Theorem~\ref{thm:caterpillar}, then $\au(T)=1$.
\end{theorem}
\begin{proof}
Let $v_{i,m}$ denote the $m$th leaf neighbor of vertex $v_i$ along the spine,
for $1\le m\le d'(v_i)$.  Let $A$ be the set of leaf neighbors of $\VEC v1k$,
so $|A|\ge \ell(k)$.  Let $B$ be the set of cells in the pyramid over
$\{\VEC v1k\}$, so $|B|=\ell(k)$.

We specify edge costs $w$ for the complete bipartite graph with parts $A$ and
$B$, by
$$w(v_{j,m}b_{i,h})=\begin{cases}|i-j| & \text{if }h>|i-j|\\ \infty &\text{if }h\le|i-j|\end{cases}.$$
By assumption, $|A|\ge |B|$.  If $|A|>|B|$, vertices can be added to $B$ (with
zero cost on incident edges) until the parts have equal size.  Let $M$ be a
perfect matching of minimum cost in the resulting graph.  Such a matching can
be obtained by the Hungarian Algorithm~\cite{K,M}.

\bigskip
{\bf Claim 1:} {\it $M$ has finite cost.}
Let $H$ be the subgraph formed by the edges of positive finite cost.
We prove that $H$ contains a matching that covers $B$.  By Hall's
Theorem~\cite{Hall}, this holds if and only if $|N_H(X)|\ge|X|$ whenever
$X\esub B$.  If if fails, then let $X\subseteq B$ be a minimal set such that
$|N_H(X)|<|X|$.

Because $X$ is minimal, it follows that the subgraph of $H$ induced by $X$ and
$N_H(X)$ is connected.  For $b_{i,h}\in B$, by definition $N_H(b)$ is the set
of leaf neighbors of the spine vertices $v_j$ such that $|i-j|<h$; these spine
vertices are consecutive along the spine.  If the union of these segments for
the vertices of $X$ is not a single consecutive segment, then again $X$ is
not a minimal failure of Hall's Condition.  Thus $N(X)=\SE jab d'(v_j)$ for
some $a$ and $b$.

By definition, all leaf neighbors of $v_i$ are in $N_H(b_{i,h})$.  Hence
each element of $X$ lies in the pyramid over $\{v_a,\ldots,v_b\}$.  Thus
$|X|\le\ell(b-a+1)$.
We now have $|N(X)|=\SE jab d'(v_j)\ge\ell(b-a+1)\ge|X|$.
Thus Hall's Condition is satisfied, and $M$ has finite cost.

\bigskip
We use $M$ to specify a protocol that converts $T$ into an ascending tree.
We fill all cells in the pyramid by moving to each cell the chip on the leaf
vertex matched to it in $M$.  This places weight $1+\mu_i$ on $v_i$ for
$1\le i\le k$.  Vertices $v_0$ and $v_{k+1}$ retain their original chips and
weight $1$.  Leaves not matched by edges with positive cost in $M$ also retain
weight $1$ in $T$.  Hence this new distribution $T'$ is an ascending tree, with
unit acquisition number $1$.

The chip from the leaf $v_{j,m}$ matched to cell $b_{i,h}$ will move $|i-j|$
steps along the spine to reach the assigned destination $v_i$, but we still
must determine a feasible order for these moves to occur in a protocal.
We begin by representing the desired moves on a grid.  Place a filled circle
for $b_{i,h}$ at the lattice point $(i,h)$.  For an edge $v_{j,m}b_{i,h}$ in
$M$, draw a line segment from $(j,h-|i-j|)$ upward along a diagonal to the
circle at $(i,h)$.  When $i=j$, the segment has length $0$ (see
Figure~\ref{fig:matcttoacqmoves}).  Each circle corresponding to a cell in the
pyramid is the top end of one segment, and the slope of each segment is $\pm1$.  

\begin{figure}[hbt]
\centering
\begin{tikzpicture}
\foreach \x in {1,...,7}
	\fill (\x,0) circle (0.15);
\foreach \x in {2,...,6}
	\fill (\x,1) circle (0.15);
\foreach \x in {3,...,5}
	\fill (\x,2) circle (0.15);
\foreach \x in {4,...,4}
	\fill (\x,3) circle (0.15);
\draw [line width=2] (1,0) -- (2,1);
\draw [line width=2] (3,1) -- (4,0) -- (5,1);
\draw [line width=2] (4,1) -- (3,2);
\draw [line width=2] (7,0) -- (4,3);
\end{tikzpicture}
\caption{A matching being converted to a protocol.\label{fig:matcttoacqmoves}}
\end{figure}
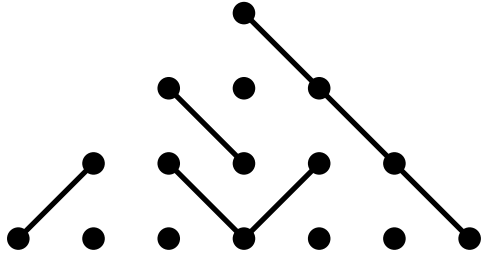

These line segments may overlap (one may contain another), but they do not
cross.

\bigskip
{\bf Claim 2:} {\it Two non-collinear segments in the grid diagram cannot
share a lattice point $(a,b)$ if one extends above $(a,b)$ and the other
extends below $(a,b)$.}
If the segment $L$ extending above contains $(a-1,b+1)$ and the segment $L'$
extending below contains $(a-1,b-1)$, then $M$ contains edges $v_{j,m}b_{i,h}$
and $v_{j',m'}b_{i',h'}$ such that $j>a\ge i$ and $j'<a\le i'$.
These edges have cost $i-j$ and $i'-j'$, but since $j'<j$ and $i\le i'$, we
have $|i-j'|+|i'-j|<(i-j)+(i'-j')$.  Hence replacing these edges in $M$ with
the edges $v_{j,m}b_{i',h'}$ and $v_{j',m'}b_{i,h}$ yields a matching $M'$
with smaller cost, contradicting the minimality of $M$.

Similarly, if the segment $L$ extending above contains $(a+1,b+a)$ and the
segment $L'$ extending below contains $(a+1,b-1)$, then $M$ contains edges
$v_{j,m}b_{i,h}$ and $v_{j',m'}b_{i',h'}$ such that $j<a\le i$ and
$j'>a\ge i'$, and the edges $v_{j,m}b_{i',h'}$ and $v_{j',m'}b_{i,h}$
have smaller total cost.

\bigskip
Finally, we convert the grid diagram to a protocol.  We have $\ell(s)$ segments
reaching the points for cells in the pyramid; they may have length $0$.
A segment $D$ is \textit{below} a segment $D'$ if $D$ and $D'$ contain points
with the same horizontal coordinate such that the point in $D$ is vertically
below the point in $D'$, or if $D$ and $D'$ are collinear and the destination
cell of $D$ has a smaller vertical coordinate than that of $D'$.  By Claim 2,
no two distinct segments can be below each other.  Linearly order the cells of
the pyramid by iteratively taking a remaining cell $b_{i,h}$ such that no
remaining cell has its segment below that of $b_{i,h}$.

%

The resulting linear order on the points is the order in which we fill the
cells to obtain the ascending tree described earlier.  When $b_{i,h}$ is to
be filled, and $v_{j,m}b_{i,h}$ is the edge incident to it in $M$, the chip
from the $m$th leaf neighbor of $v_j$ is moved to $v_i$.  Since all segments
below the edge for this segment have been processed, and this segment is not
below any edge that has been processed, the current weights on the spine
vertices from $v_j$ to $v_i$ permit this chip to move as desired.
Hence we convert $T$ to an ascending tree, and $\au(T)=1$.
\end{proof}

Theorem~\ref{thm:caterpillar} yields a linear-time algorithm for $\au(T)$ on
caterpillars.

\begin{corollary}
There is a $O(|V(T)|)$-time algorithm that determines the unit acquisition number of a caterpillar $T$.
\end{corollary}

\begin{proof}
Let $T$ be a caterpillar, drawn with the spine in order from left to right. 
The algorithm iteratively removes the largest caterpillar subtree having unit
acquisition number $1$ that contains the left end of the remaining spine.
This caterpillar is determine by applying the condition in
Theorem~\ref{thm:caterpillar}.

When removing a subtree $T'$, it may be necessary to count a member of the
spine of $T$ belonging to $T'$ as a leaf of $T'$ in some cases: 1) the leftmost
member of $T_i$ has no leaf neighbors, 2) the rightmost member of $T_i$ has no
leaf neighbors, or 3) the two rightmost members of $T_i$ have no leaf neighbors
(see Figure~\ref{Fig:catadjust}).

\begin{figure}[hbt]
\begin{center}
\begin{tikzpicture}
\node at (-2,.5) {$1$)};
\fill (0,1) circle (0.15);
\fill (1,1) circle (0.15);
\fill (.7,0) circle (0.15);
\fill (1.3,0) circle (0.15);
\draw [dotted, line width=1.5] (-.5,1)--(0,1);
\draw [line width=1.5] (0,1)--(2,1);
\draw [line width=1.5] (.7,0)--(1,1)--(1.3,0);

\fill (6,0) circle (0.15);
\fill (6,1) circle (0.15);
\fill (5.6,0) circle (0.15);
\fill (6.4,0) circle (0.15);
\draw [line width=1.5] (6,0)--(6,1)--(7,1);
\draw [line width=1.5] (5.6,0)--(6,1)--(6.4,0);

\node at (-2,-1.5) {$2$)};
\fill (0,-1) circle (0.15);
\fill (-.3,-2) circle (0.15);
\fill (.3,-2) circle (0.15);
\fill (1,-1) circle (0.15);
\fill (2,-1) circle (0.15);
\fill (1.7,-2) circle (0.15);
\fill (2.3,-2) circle (0.15);
\draw [line width=1.5] (-1,-1)--(3,-1);
\draw [line width=1.5] (-.3,-2)--(0,-1)--(.3,-2);
\draw [line width=1.5] (1.7,-2)--(2,-1)--(2.3,-2);

\fill (6,-1) circle (0.15);
\fill (5.6,-2) circle (0.15);
\fill (6.4,-2) circle (0.15);
\fill (6,-2) circle (0.15);
\fill (8,-1) circle (0.15);
\fill (7.7,-2) circle (0.15);
\fill (8.3,-2) circle (0.15);
\draw [line width=1.5] (5,-1)--(6,-1)--(6,-2);
\draw [line width=1.5] (5.6,-2)--(6,-1)--(6.4,-2);
\draw [line width=1.5] (7.7,-2)--(8,-1)--(8.3,-2);
\draw [line width=1.5] (8,-1)--(9,-1);

\node at (-2,-3.5) {$3$)};
\fill (0,-3) circle (0.15);
\fill (-.3,-4) circle (0.15);
\fill (.3,-4) circle (0.15);
\fill (1,-3) circle (0.15);
\fill (2,-3) circle (0.15);
\fill (3,-3) circle (0.15);
\fill (2.7,-4) circle (0.15);
\fill (3.3,-4) circle (0.15);
\draw [line width=1.5] (-1,-3)--(4,-3);
\draw [line width=1.5] (-.3,-4)--(0,-3)--(.3,-4);
\draw [line width=1.5] (2.7,-4)--(3,-3)--(3.3,-4);

\fill (6,-3) circle (0.15);
\fill (5.7,-4) circle (0.15);
\fill (6.3,-4) circle (0.15);
\fill (7,-3) circle (0.15);
\fill (7,-4) circle (0.15);
\fill (9,-3) circle (0.15);
\fill (8.7,-4) circle (0.15);
\fill (9.3,-4) circle (0.15);
\draw [line width=1.5] (5,-3)--(7,-3)--(7,-4);
\draw [line width=1.5] (9,-3)--(10,-3);
\draw [line width=1.5] (5.7,-4)--(6,-3)--(6.3,-4);
\draw [line width=1.5] (8.7,-4)--(9,-3)--(9.3,-4);
\end{tikzpicture}
\caption{Situations where a spine vertex is treated as a leaf in a subtree.}
\label{Fig:catadjust}
\end{center}
\end{figure}
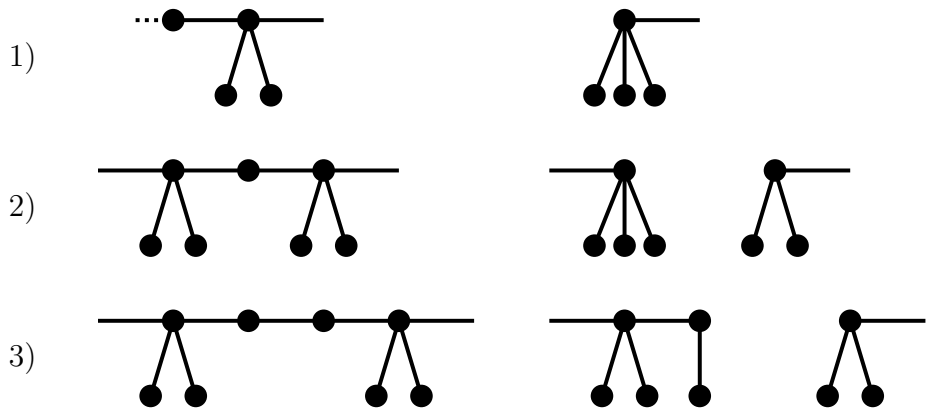

The algorithm partitions the vertex set of $T$ into set inducing caterpillars
with unit acquisition number $1$, so it provides an upper bound for $\au(T)$.
If $\au(T)$ is less than the bound provided by the algorithm, then an optimal
protocol collects all the weight at $\au(T)$ vertices.  The weight on each one
of these vertices comes from a subtree of $T$ that is a caterpillar.  Since
this partitions $V(T)$ into fewer sets inducing caterpillars, and the spines
of these caterpillars and those found by the algorithm, some caterpillar $T^*$
used by the optimal protocol properly contains some caterpillar $T'$ found by
the algorithm, extending farther on {\it both} ends.  Now the necessity of the
condition in Theorem~\ref{thm:caterpillar} for $T^*$ implies that the condition
also holds for the extension of $T'$ to the right end of $T^*$, and then
sufficiency of the condition contradicts the choice of $T'$ to be a largest
caterpillar from its left end.

To run the algorithm, we first compute the values $d'(v)$ for vertices along
the spine, in time linear in $|V(T)|$.  To understand the subsequent running
time, consider the extraction of the first caterpillar.  The condition of
Theorem~\ref{thm:caterpillar} must be checked.  Since the condition is
specified over all segments of internal vertices of the subtree, it also holds
for all sub-segments.  That is, we grow the potential caterpillar from the left.
If the condition holds for the first $i-1$ internal vertices of the spine, then
we next test the segments ending at the $i$th internal vertex.  At the point
where the test first fails, the number of sums that have been tested is
quadratic in the length of the spine examined so far, but the number of leaves
adjacent to those spine vertices is also quadratic in that length.  Thus, the
number of sums performed is linear in the number of vertices that are collected
by the first caterpillar.  Since this holds for each subcaterpillar, the 
entire algorithm runs in time linear in $|V(T)|$.
\end{proof}

\section{Diameter $2$}\label{Sec:diam2}

In this section we prove $\au(G)\le2$ when $G$ has diameter $2$, with equality
possible only for $C_5$ and the Petersen graph.  It remains open whether this
also holds for $\at(G)$.

We begin with a lemma.  A {\it solo-neighbor} of a vertex $u$ in a clique $Q$
with at least two vertices is a vertex outside $Q$ whose only neighbor in $Q$
is $u$.

\begin{lemma}\label{NoSoloNbr}
Let $G$ be a graph with diameter $2$.  If $Q$ is a clique in $G$ with at least
two vertices, and there exists $u\in Q$ such that $u$ has no solo-neighbor,
then $a_{u}(G) = 1$.
\end{lemma}

\begin{proof}
Given such a vertex $u$, let $v$ be another vertex in $Q$.  Since $u$ has no
solo-neighbor relative to $Q$, every vertex has a path of length at most $2$
to $v$ that does not pass through $u$.  Thus moving $c_u$ to $v$ converts
$G$ to a graph with a spanning ascending tree.
\end{proof}

The {\it girth} of a graph $G$ having a cycle is the minimum length of a cycle
in $G$.

\begin{theorem}
If $G$ is a graph with diameter $2$, then $\au(G)\le2$.  Equality can only
hold for $C_5$ or possibly the Petersen graph.  Otherwise, $a_{u}(G) = 1$.
\end{theorem}

\begin{proof}
A graph with diameter $2$ that is not a star has girth at most $5$.  We
consider cases.

\textbf{Case 1:}  {\it $G$ is $C_{5}$ or the Petersen graph.}
Proposition~\ref{partpath} shows that $\au(C_{5}) =2$.  Consider the Petersen
graph $P$.  For the upper bound, move weight along a perfect matching onto a
$5$-cycle, with weight $2$ at each vertex.  Now $\au(C_5)=2$ suffices.

In a tree $T$ with $\au(T)=1$, every non-leaf vertex $v$ must have a leaf
neighbor; otherwise, the first move that involves $v$ leaves a cut-vertex with
weight $0$ separating chips.  A short case analysis shows that the only such
$10$-vertex tree with maximum degree at most $3$ is that in
Figure~\ref{fig:tree}.  This tree satisfies Theorem~\ref{thm:caterpillar} but
is not contained in $P$.

Therefore, if $\au(P)=1$, then the edges used by acquisition moves include a
cycle.  A lengthy case analysis shows that in such a protocol also cannot
move all weight to a single vertex.  Thus $\au(P)=2$ is provable, but we omit
this analysis.


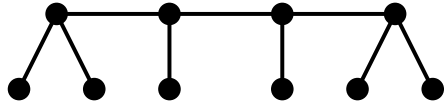
\begin{figure}[hbt]
\begin{center}

\begin{tikzpicture}
\fill (0,0) circle (0.15);
\fill (1,0) circle (0.15);
\fill (.5,1) circle (0.15);
\fill (2,1) circle (0.15);
\fill (2,0) circle (0.15);
\fill (3.5,1) circle (0.15);
\fill (3.5,0) circle (0.15);
\fill (5,1) circle (0.15);
\fill (4.5,0) circle (0.15);
\fill (5.5,0) circle (0.15);

\draw [line width=1.5] (0,0) -- (.5,1) -- (5,1) -- (5.5,0);
\draw [line width=1.5] (1,0) -- (.5,1);
\draw [line width=1.5] (2,0) -- (2,1);
\draw [line width=1.5] (3.5,0) -- (3.5,1);
\draw [line width=1.5] (4.5,0) -- (5,1);
\end{tikzpicture}

\caption{The only $10$-vertex tree $T$ with maximum degree $3$ such that
$\au(T)=1$.}\label{fig:tree}
\end{center}
\end{figure}

\textbf{Case 2:}  {\it $G$ has girth $3$.}
Let $Q$ be a maximum clique in $G$, so $|Q|\geq 3$.
Let $S$ be the set of vertices outside $Q$ having a neighbor in $Q$,
and let $U=V(G)-Q-S$.  If $U=\nul$, then $\au(G) = 1$.  If some vertex $u$ in
$Q$ has no solo-neighbor, then Lemma~\ref{NoSoloNbr} yields $\au(G) = 1$.

In the remaining case, $U\ne\nul$, and every vertex of $Q$ has a solo-neighbor.
Choose $u,v,w\in Q$.  Let $S'$ be the set of solo neighbors of vertices
in $Q-\{w\}$.  If $N_G(z)\esub S'$ for some $z\in U$, then $w$ has distance $3$
from $z$.  Hence $U$ has no such {\it bad} vertex.

Begin by moving all weight from solo-neighbors of $v$ to $v$.  Next move
all weight from solo-neighbors of $u$ to $v$, two moves per chip; note that
now $\wt(v)\ge3$.  Next move $c_u$ to $w$, and let each vertex of $Q-\{u,v,w\}$
acquire the chip from one of its solo-neighbors.  Now $\wt(v)\ge3$, all
vertices of $Q-\{u,v\}$ have weight $2$, other vertices retaining positive
weight have weight $1$, and every vertex with weight $1$ has a neighbor in
$Q-\{u\}$ or a neighbor of weight $1$ with a neighbor in $Q-\{u\}$ (because $U$
has no bad vertex).  Hence positive weight remains only on an ascending tree
rooted at $v$, and $\au(G)=1$.

\textbf{Case 3:} {\it $G$ has girth $4$.}
Let $w,x,y,z$ be the vertices of a $4$-cycle in order.  Move $c_w$ to $x$ and
$c_z$ to $y$.  We claim that all the weight is now contained in disjoint
ascending trees $T_x$ and $T_y$ to $x$ and $y$.  If so, then $x$ and $y$ can
acquire all the weight, after which it can be combined along the edge $xy$.

Let $X=N_G(x)-\{y,w\}$ and $Y=N_G(y)-\{x,z\}$.  The sets $X$ and $Y$ are
disjoint.  Let $u$ be a vertex with weight $1$ not in $X\cup Y$.  If $u$ has no
neighbor in $X$ or $Y$, then distance at most $2$ from both $x$ and $y$
requires $u\in N_G(w)\cap N_G(z)$.  Now $u,w,z$ form a $3$-cycle, which is
forbidden.  Hence $u$ has a neighbor in $X$ or $Y$, and the two desired trees
exist.

\textbf{Case 4:} {\it $G$ has girth $5$.}
When $G$ has minimum degree $k$, diameter $2$ limits $G$ to $1+k^2$ vertices.
In particular, when $G$ is not $C_{5}$ or the Petersen graph, diameter $2$
requires $\delta(G)\ge 4$.  Let $v,w,x,y,z$ in order be the vertices of a
$5$-cycle $C$ in $G$.  Every vertex outside $V(C)$ has at most one neighbor
on $C$, since $G$ has girth $5$.  Let $V=N_G(v)-V(C)$, and similarly define
$W,X,Y,Z$.  Let $N$ be the set of vertices with no neighbors on $C$.

Since $\delta(G)\ge4$, each of $V,W,X,Y,Z$ has size at least $4$.
Hence we can move chips from two vertices of $X$ to $x$, giving $x$ weight $3$.
Also, diameter $2$ requires every vertex of $Y$ to have a neighbor in $W$ in
order to reach $w$.  Hence there is an edge $w'y'$ with $w'\in W$ and $y'\in Y$.
Move $c_{w'}$ to $w$ and $c_{y'}$ to $y$, giving weight $2$ to $w$ and $y$.

We claim that now all weight in $G$ lies in an ascending tree rooted at $x$,
so $\au(G)=1$.  Since the paths to $x$ along $C$ strictly ascend in weight, the
claim holds immediately for all vertices except those in $N$.  To reach $w$ and
$y$ in two steps, each vertex of $N$ must be adjacent to one vertex in each of
$W$ and $Y$ (exactly one, to avoid $4$-cycles).  However, avoiding $3$-cycles
requires that no vertex of $N$ is adjacent to both $w'$ and $y'$.  Hence each
vertex of $N$ has a neighbor in $W$ or $Y$ with weight $1$.  Since $w$ and $y$
have weight $2$, the claim holds.
\end{proof}

\section*{Acknowledgments}
The authors would like to thank Noah Prince for helpful discussions on this topic.

\end{document}